\documentclass[11pt]{article}

\usepackage[utf8]{inputenc}
\usepackage[T1]{fontenc}

\usepackage{a4wide}
\usepackage{amsmath}
\usepackage{amsthm}
\usepackage{amssymb}
\usepackage{subfigure}
\usepackage{enumerate}
\usepackage{proof}
\usepackage{url}

\usepackage[dvipsnames]{color}
\usepackage{tikz}

\newtheorem{lemma}{Lemma}
\newtheorem{proposition}[lemma]{Proposition}
\newtheorem{theorem}[lemma]{Theorem}
\newtheorem{corollary}[lemma]{Corollary}
\newtheorem{claim}{Claim}

\def\zp{{\mathbb Z}/p{\mathbb Z}}

\newcommand{\subgp}[1]{\langle{#1}\rangle}

\begin{document}

\title{A Structure Theorem for Small Sumsets in Nonabelian Groups}

\author{ Oriol Serra, Gilles Z\'emor}

\date{February 28, 2012}
\maketitle

\begin{abstract}
Let $G$ be an arbitrary finite group and let $S$ and $T$ be two subsets
such that $|S|\geq 2$, $|T|\geq 2$, and $|TS|\leq |T|+|S|-1\leq
|G|-2$.
We show that if $|S|\leq |G|-4|G|^{1/2}$ then either $S$ is a
geometric progression or there exists a non-trivial subgroup $H$ such
that either $|HS|\leq |S|+|H|-1$ or $|SH|\leq |S|+|H|-1$. 
This extends to the nonabelian case classical reults for abelian groups.
When we
remove the hypothesis $|S|\leq |G|-4|G|^{1/2}$ we show the existence
of counterexamples to the above characterization
whose structure is described precisely.
\end{abstract}

\section{Introduction}

Let $(G,+)$ be a finite abelian group written additively. Let $S$ be a subset of $G$
such that $T+S\neq G$ and
\begin{equation}
  \label{eq:-2}
  |T+S|\leq |T|+|S|-2
\end{equation}
for some subset $T$ of
$G$. A Theorem of Mann \cite{MAN53} says that $S$ must be well covered
by   cosets of a subgroup. More precisely, there must exist a
proper subgroup $H$ of $G$ such that
  $$|S+H| \leq |S| + |H| -2.$$
Mann's Theorem can be thought of as simplified, or one-sided, version
of Kneser's Theorem \cite{KNE} which gives a structural result for the pair of subsets
$\{S,T\}$ rather than a single subset. If one weakens the condition
\eqref{eq:-2} to
\begin{equation}
  \label{eq:-1}
  |T+S|\leq |T|+|S|-1
\end{equation}
for some set $T$ such that $|S+T|\leq |G|-2$, then a structural change
occurs because the sets $S$ and $T$ can be arithmetic progressions
and not well covered by cosets. However, this is the only alternative
i.e. if $|T|\geq 2$ and $S$ is not an arithmetic progression, then a simple,
one-sided, version of the Kempermann Structure Theorem \cite{KEM} says that there
must exist a proper subgroup such that
\begin{equation}
  \label{eq:H-1}
  |S+H| \leq |S| + |H| -1.
\end{equation}

In the present work we are interested in the nonabelian counterpart of
the above results. Caution is in order because the two-sided abelian
additive theorems do not seem to generalize. In particular
counter-examples to the intuitive nonabelian generalization of
Kneser's Theorem were found by Olson \cite{ols86} and the second
author \cite{z94}. However,
Mann's theorem was generalized to the nonabelian
setting \cite{z94,yh99}. It was obtained that, if $S$ is a subset of a
finite group $(G,\times)$ (from now on written multiplicatively to
emphasize that $G$ is not necessarily abelian)
for which there is a subset $T$ such that
$TS\neq G$ and
$$|TS|\leq |T|+|S|-2,$$
then there must exist a proper subgroup $H$ such that $S$
is well-covered by either left or right cosets modulo a subgroup $H$,
i.e. we have
$$\text{either}\; |SH| \leq |S| + |H| -2\;\text{or}\;
  |SH| \leq |H| + |S| -2.$$
Note that the difference with the abelian case is that we cannot
control whether $S$ is covered by left or right cosets.

Our main result is to obtain a structural result on $S$ under a
generalization of \eqref{eq:-1} to nonabelian groups. Specifically, we
prove:
\begin{theorem}\label{thm:main}
  Let $S$ be a subset of a finite group $G$ for which there exists a
  subset $T$ such that $2\leq |T|$ and $|TS|\leq \min(|G|-2,|T|+|S|-1)$.
  Then one of the following holds
  \begin{enumerate}[\it (i)]
  \item $S$ is a geometric progression, i.e. there exist $g,a\in
    G$ such that one of the two sets $gS,Sg$ equals $\{1,a,a^2,\ldots
    ,a^{|S|-1}\}$,
  \item there exists a proper subgroup $H$ of $G$ such that
    $$|HS^{\varepsilon}| \leq |H|+|S|-1$$
  where $S^{\varepsilon}$ denotes either $S$ or $S^{-1}$.
  \item there exists a subgroup $H$ and an element $a$
    of $G$ such that $|HaH|=|H|^2$ and, letting $A=H\cup Ha$,
    $$|AS^{\varepsilon}| = |A| + |S|-1 = |G|- |A|.$$
  \end{enumerate}
\end{theorem}

Note that property $(iii)$ collapses to a particular case of $(i)$ if
the group $G$ is abelian, since then we can only have
$H=\{1\}$. Condition $|HaH|=|H|^2$ in $(iii)$ also implies that it can
only occur for subsets $S$ of $G$ that are quite close to being the
whole group, since we must clearly have $|H|\leq |G|^{1/2}$ and
$|S|=|G|+1-4|H|$, in other words:

\begin{corollary}\label{cor:sqrt}
  Let $S$ be a subset of a finite group $G$ for which there exists a
  subset $T$ such that $2\leq |T|$ and $|TS|\leq
  \min(|G|-2,|T|+|S|-1)$ and such that $|S|\leq |G|-4|G|^{1/2}$: then
  \begin{itemize}
  \item either $S$ is a geometric progression,
  \item or there exists a proper subgroup $H$ of $G$ such that
    $$|HS^{\varepsilon}| \leq |H|+|S|-1.$$
  \end{itemize}
\end{corollary}

The condition $|S|\leq |G|-4|G|^{1/2}$ in Corollary~\ref{cor:sqrt}
is unlikely to be improved upon asymptotically, for we shall show
in the final section that, assuming a number-theoretic conjecture
(the existence of an infinite number of Sophie Germain primes),
there exist infinite
families of groups $G$ with subsets $S$ such that $|S| \leq
|G|-O(\sqrt{|G|})$
and that satisfy the hypothesis of Corollary~\ref{cor:sqrt} but not
its conclusion.

We shall use Hamidoune's atomic method to derive
Theorem~\ref{thm:main}. If $S$ is a generating subset containing $1$
of a finite group, then $A$ is a $k$-atom of $S$ if it is of minimum
cardinality among subsets $X$ such that $|X|\geq k$, $|XS|\leq |G|-k$
and $|XS|-|X|$ is of minimum possible cardinality (see
Section~\ref{sec:prelim} for detailed definitions). The isoperimetric
or atomic method was introduced by Hamidoune in \cite{yh96} and
the study of $k$-atoms was used in a number of papers
including \cite{hsz06,hsz08} to give structural
results in abelian groups on sets $S,T$ such that $|S+T|\leq |S|+|T|+m$.
The study of $k$-atoms
for $k=1,2$ has also yielded generalizations to nonabelian groups of several
addition theorems \cite{z94,yh96,hls98,yh99,hsz07}.
The structure of $2$-atoms is again crucial to our present study.
Under the hypothesis of Theorem~\ref{thm:main},
when a $2$-atom of $S$ or $S^{-1}$ has cardinality $2$, it yields
property $(i)$. When a $2$-atom of $S$ or $S^{-1}$ is a subgroup, it
yields property $(ii)$. It follows from results in \cite{yh00} that in
the abelian case, and in some particular nonabelian cases, under
the hypothesis of Theorem~\ref{thm:main}, any
$2$-atom containing the unit element must either have cardinality $2$
or be a subgroup. However, deriving a similar result in
the general nonabelian case has not been  attempted  until recently, when the topic
has attracted the attention of a number of researchers in the area of so--called approximate groups, see e.g. \cite{taoblog}.

In one of his last preprints
\cite{yh09}, Hamidoune shows that, under the hypothesis of
Theorem~\ref{thm:main}, if $(i)$ and $(ii)$ do not hold then any
$2$-atom of $S$ containing $1$ must be of the form $A=H\cup Ha$ for
some subgroup $H$. He does not try to find out however, whether these
particular $2$-atoms can actually exist, and our original contribution
is to show that, somewhat surprisingly, this last case can indeed
occur, but only just, namely only if $HS^\varepsilon$ consists of exactly all but two right
$H$-cosets of $G$.

The paper is organised as follows: in the Section \ref{sec:prelim} we introduce
the isoperimetric tools that we shall need in the sequel. We then
proceed in Section \ref{sec:nonperiodic} to obtain a general upper bound
for the cardinality of nonperiodic $2$--atoms, thus extending an analogous result
obtained by Hamidoune \cite{yh00} for abelian groups, for normal sets
in simple groups by Arad and Muzychuk \cite{am97} and in \cite{bps11}
for torsion-free groups. Section \ref{sec:periodic} contains
a somewhat shortened account of Hamidoune's result on $2$-atoms
obtained in \cite{yh09}. We then go on to study the particular case
of $2$-atoms equal to the union of two cosets with the purpose of
showing that this last case can mostly not exist, except in an
exceptional degenerate case. Our main method will be to show that
the exceptional $2$-atom defines an edge-transitive graph on cosets
of a subgroup of $G$ and a close study of the edge-connectivity of
this graph will rule out all subsets $S$ but the ones described above.
We conclude by exhibiting a family of examples that show that the leftover case can
indeed occur.

\section{Notation and preliminary results}\label{sec:prelim}

Let $G$ be a finite group and let $S\subset G$ be a generating subset
containing the unit element $1$. For $X$ a subset of $G$ we shall
write
$$\partial_SX = XS\setminus X$$
and $X^{*S} = G\setminus (X\cup \partial_SX)$. When the set $S$ is
implicit and no confusion can occur we simply write $\partial X$ and $X^*$.

We shall say that $S$ is {\it
  $k$-separable} if there exists $X\subset G$ such that $|X|\geq k$
and $|X^*|\geq k$. Suppose that $S$ is $k$-separable. The {\it $k$-th
isoperimetric number} of $S$ is defined by
  $$\kappa_k(S) = \min\{\partial X\; |\; |X|\geq k, |X^*|\geq k\}.$$
For a $k$-separable set $S$, a subset $X$ achieving the above minimum
is called a $k$-fragment of $S$. A $k$-fragment with minimal
cardinality is called a $k$-atom.

{\bf Comments.}
The above definitions can be formulated, as usually done by Hamidoune, 
in the more general context of
directed graphs \cite{yh00} where sets $X$ are subsets of vertices,
$\partial X$ (the boundary of $X$) is the set of neighbouring vertices
of $X$ not in $X$, and $X^*$ is the complement i.e. the set of
vertices neither in $X$ nor $\partial X$. We are dealing with the case
of a Cayley graph $G$ with set of generators equal to the non unit
elements of $S$. The following properties are straightforward and will
be used throughout.
\begin{itemize}
\item If $A\subseteq B$ then $B^*\subseteq A^*$.
\item If $S$ is $k$--separable, then $1\leq \kappa_{k-1}(S)\leq
  \kappa_k(S)$. We also have $\kappa_k(S)=\kappa_k(S^{-1})$.
\item For any $s\in S^{-1}$, $\kappa_k(S) =\kappa_k(Ss)$ and $k$--fragments
  (resp. $k$--atoms) of $S$ are $k$--fragments (resp. $k$--atoms) of $Ss$.
\item If $F$ is a $k$--fragment of $S$, then $F^*$ is a $k$--fragment of
  $S^{-1}$.
\item If $F$ is a $k$--fragment (resp. $k$-atom) of $S$ then every left
  translate $gF$ is a
  $k$-fragment (resp. $k$--atom) of $S$ for any group element $g$.
\end{itemize}

Finally we denote by $\alpha_k(S)$ the size of a $k$--atom of $S$.

The following two theorems obtained by Hamidoune \cite{yh00} are basic
pieces of the atomic method. The first two are intersection properties
of atoms.

\begin{theorem}[The intersection property for atoms]\label{thm:int}
  Let $S$ be a $k$--separable subset of the group $G$. Let $A$
  and $B$ be two distinct $k$--atoms of $S$.
    If $|G|\ge 2\alpha_k(S)+\kappa_k(S)$, then $|A\cap B|\leq k-1$.
\end{theorem}

Note that if $\alpha_2(S)\leq\alpha_2(S^{-1})$ then we have 
$|G|\ge 2\alpha_k(S)+\kappa_k(S)$ and the intersection property holds
for atoms of $S$. Therefore, if the intersection property does not
hold for the atoms of $S$ then it holds for the atoms of $S^{-1}$.
If $\alpha_2(S)\leq\alpha_2(S^{-1})$ the following more general result holds:

\begin{theorem}[The fragment-atom intersection property]\label{thm:intAF}
  Let $S$ be a $k$--separ\-able subset of the group
  $G$. Suppose $\alpha_2(S)\leq \alpha_2(S^{-1})$ and let $A$
  and $F$ be a $k$--atom and a $k$--fragment of $S$ respectively.
  Then
   \begin{itemize}
  \item either $A\subset F$
  \item or $|A\cap F|\leq k-1$.
  \end{itemize}
\end{theorem}

The intersection
Theorem~\ref{thm:intAF} together with the fact that left translates of an atom are
atoms implies that either $1$--atoms containing $1$ of $S$ or $1$--atoms
containing $1$ of $S^{-1}$ are subgroups.

\begin{theorem}[\cite{yh84}]\label{thm:1atom}
Let $S$ be a $1$--separable generating set of a group $G$ with $1\in
S$ and $\alpha_1 (S)\le \alpha_1 (S^{-1})$. The  atom of $S$ containing $1$ is a subgroup
\end{theorem}

Note that without loss of generality we can add to the hypothesis
of Theorem~\ref{thm:main} that $S$ contains $1$ (if not replace $S$
by a right translate of $S$) and that $S$ generates $G$ (if $S$ is not
$1$-separable in the subgroup that it generates, then part $(ii)$ of 
Theorem~\ref{thm:main} holds trivially).
In that case the hypothesis of Theorem~\ref{thm:main} translates to:
$\kappa_2(S)\leq |S|-1$.

In the abelian case it was proved in \cite{yh00} (see also
\cite{hsz08}) that $2$-atoms that
are not subgroups have cardinality at most $\kappa_2(S)-|S|+3$. This
implies in particular that if $\kappa_2(S)\leq |S|-1$ then $2$-atoms are
either subgroups or of cardinality $2$. In turn this gives that under
the hypothesis \eqref{eq:-1} either $S$ is an arithmetic progression
($2$-atoms are of cardinality $2$) or $2$-atoms are subgroups which
yields \eqref{eq:H-1}.

In the general, non abelian case, it was obtained in \cite{am97} in
the special case of simple groups and normal sets $S$
that the cardinality of $2$-atoms is at most $\kappa_2(S)-|S|+3$.
In the next section we prove in all generality for all finite groups,
that if $|G|\ge 2\alpha_2(S)+\kappa_2 (S)$, then either
$\kappa_2(S)-|S|+3$ or $2$-atoms $A$ of $S$ are {\em left-periodic}, meaning
that there exists group elements $x\neq 1$ such that $xA=A$.
Note once more that if the condition 
$|G|\ge 2\alpha_2(S)+\kappa_2(S)$
does not hold for $S$, then it necessarily holds for $S^{-1}$.

\section{Nonperiodic $2$--atoms}\label{sec:nonperiodic}

In what follows $G$ is a finite group and $S$ is a $2$--separable
generating set of $G$ with $1\in S$ and $|S|\ge 3$. 
We moreover assume that $|G|\ge 2\alpha_2(S)+\kappa_2 (S)$.
The purpose of this
section is to prove  that a $2$--atom $U$ of $S$ which is not
left--periodic has cardinality at most  $|U|\le \kappa_2 (S)-|S|+3$, 
Proposition \ref{prop:m+3} below.

\begin{lemma}\label{lem:rightper} Let $A$ be a $2$--atom of $S$ and
  let $H\subset G$ be its maximal left period, i.e. the maximal
  subgroup such that $HA=A$. Then,
\begin{equation}\label{eq:intperr}
|A\cap Ag|\le |H|\; \mbox{ for all }\; g\in G\setminus \{ 1\}.
\end{equation}
In other words, $A\cap Ag$ contains at most a single coset. In
particular, if $g\in H\setminus \{ 1\}$ then 
$A\cap Ag\subseteq H$.
\end{lemma}

\begin{proof}  The statement trivially holds if $A=H$. Suppose that
$$
A=Ha_1\cup\cdots \cup Ha_t,
$$
is the union of $t\ge 2$ different cosets of $H$, where $a_1=1$. Since
$A\neq gA$ for each $g\not\in H$, 
the intersection property of $A$ now says
\begin{equation}\label{eq:intper}
|A\cap gA|\le 1 \mbox{ for all } g\in G\setminus H.
\end{equation}

Consider first that $g\in H\setminus \{ 1\}$. By the intersection
property \eqref{eq:intper} of $A$  we have $|A\cap a_iA|\le 1$ for
each $i>1$. 
According to the decomposition  of $A$ into right cosets of $H$, this means
\begin{equation}\label{eq:lemk2s:4}
a_iH\cap Ha_i=\{ a_i\} \mbox{ and } a_iH\cap Ha_j=\emptyset, \; \mbox{ for } j\neq i.
\end{equation}
which implies
$$
Ha_ig=Ha_j \; \mbox{ if and only if } a_i=a_j=a_1,
$$
because $a_iga_j^{-1}\in a_iHa_j^{-1}\cap H$ and $g\neq 1$. 
It follows that $A\cap Ag=H$.

Suppose now that $g\in G\setminus H$.  Let $A=A_1\cup \ldots \cup
A_s$ be the decomposition of $A$ into maximal right
$g$--progressions, 
$A_i=\{w_i, w_ig, \ldots , w_ig^{\ell_i-1}\}$. We may assume $|A_1|\ge |A_i|$ for each $i$ and $w_1=1$.  By the intersection property of $A$ we have
$$
|A_1\cap A_1g|=|A_1\cap gA_1|\le |A\cap gA|\le 1,
$$
which implies $|A_1|\le 2$.   For each $A_i$ such that $w_i\in G\setminus H$ we have
$$
|A_1\cap  w_i^{-1}A_i|\le |A\cap w_i^{-1}A|\le 1,
$$
which implies $|A_i|=1$. Hence $(A\setminus H)g \cap A=\emptyset$ and
thus $A\cap Ag\subseteq Hg$. This completes the proof.
\end{proof}

\begin{corollary}\label{cor:left} 
Let $A$ be a  $2$-atom of $S$. If $A$ is not left--periodic then, 
for each $g\in G\setminus \{1\}$ we have
$$
\max\{ |A\cap Ag|, |A\cap gA|\}\le 1.
$$
\end{corollary}

\begin{proof}
  That $|A\cap gA|\leq 1$ is the intersection property. That
  $|A\cap Ag|\leq 1$ is Lemma~\ref{lem:rightper} since if $A$ is not
  left--periodic then its maximal period is $H=\{1\}$.
\end{proof}

\begin{lemma}\label{lem:induc} 
Let $A$ be  a $k$--atom of $S$, $k\le 2$. If $A$ is  not
left--periodic then $|A|\le \min\{2,|S|-1\}$.
\end{lemma}

\begin{proof} Suppose that $|A|>2$. 
Note that, for every element $a\in A$ there is $s\in S\setminus \{1\}$
such that $as^{-1}\in A$, 
since otherwise $|(A\setminus \{a\})S|-|A\setminus \{a\}|=|AS|-|A|$ 
contradicting the minimality of the $k$--atom.

We may therefore define a map $f:A\to S\setminus \{1\}$ which assigns to each $a$ an
$s\neq 1$ such that $as^{-1}\in A$. This map is injective
otherwise $|A\cap As^{-1}|\geq 2$, contradicting Corrollary~\ref{cor:left}.
\end{proof}

Our last preliminary step shows that an aperiodic  $2$--atom with more than two elements (actually any set satisfying the intersection property) has a large $2$--connectivity.

\begin{lemma}\label{lem:sidon} 
Let $A$ be  a $2$--atom of $S$ with $1\in A$ and $|A|\ge 3$. 
If $A$ is  not left--periodic then $A$ is $2$--separable, 
and $\kappa_2 (A)=2|A|-3$. In particular, $\kappa_1 (A)=|A|-1$.
\end{lemma}

\begin{proof} 
Let $K$ be the subgroup generated by $A$. 
Since $|A|\ge 3$ there are two distinct elements $a,a'\in K\setminus
\{1\}$. 
Thus, by the intersection property of $A$,
$$
|K|\ge |\{1,a,a'\}A|\ge |A|+|aA\setminus A|+|a'A\setminus (A\cup aA)|\ge 3|A|-3.
$$
Similarly, we have   $|XA|-|X|\ge 2|A|-3$ for each subset $X\subset K$ with cardinality~$2$. Moreover there is equality if $1\in X$ and $X\subset A$. By choosing such an $X$ we have
$$|K|-|XA|\ge (3|A|-3)-(2|A|-1)=|A|-2.$$
Therefore $A$ is $2$--separable unless $|A|=3$ and $|K|=6$. In this
case, however, the six left translates of $A$ form a system of triples
on a set of six points, every two of which intersect in at most one
point. 
Such a structure does not exist. Hence $A$ is $2$--separable and 
$\kappa_2(A)\le |AX|-|X|=2|A|-3$. We next show that $\kappa_2(A)=2|A|-3$.

Suppose on the contrary that $\kappa_2 (A)\le 2|A|-4$. Let $B$ be a $2$--atom of
$A$ containing~$1$. 
Since we have $\kappa_2 (A)=\kappa_2 (A^{-1})$, we may assume that 
$\alpha_2 (A)\le \alpha_2 (A^{-1})$, otherwise proceed by replacing
$A$ by $A^{-1}$ after noticing that $A^{-1}$ must also satisfy 
Corollary~\ref{cor:left}. We have $|B|\ge 3$ and
\begin{equation}\label{eq:lemk2s:1}
3|A|-3\le |BA|\le 2|A|+|B|-4,
\end{equation}
where the leftmost inequality uses the intersection property of $A$ and the rightmost one the fact that $B$ is a $2$--atom of $A$. This implies
\begin{equation}\label{eq:lemk2s:2}
|A|+1\le |B|,
\end{equation}
which improves \eqref{eq:lemk2s:1} to
\begin{equation}\label{eq:lemk2s:1b}
|A|(|A|+1)/2\le |BA|\le 2|A|+|B|-4,
\end{equation}
and \eqref{eq:lemk2s:2} to
\begin{equation}\label{eq:lemk2s:2b}
|A|(|A|-3)/2+4\le |B|.
\end{equation}
Moreover,  by Lemma \ref{lem:induc} applied to $B$ and $A$, the
$2$--atom $B$ is left--periodic.
Let $H\subset K$ be the stabilizer of $B$ by left translations. The subgroup $H$ is nontrivial and  $B$ is a union of right--cosets of $H$.

Suppose that $B=H$. Since $A$ generates $K$ there is an $a\in
A\setminus H$ and therefore
\begin{equation}\label{eq:lemk2s:3}
|BA|\ge |B\cup Ba|=2|B|.
\end{equation}
By inserting this inequality in the left hand side of
\eqref{eq:lemk2s:1} 
we get $|B|\le 2|A|-4$, contradicting \eqref{eq:lemk2s:2b}. 
Therefore we may assume that
$$
B=Hb_1\cup\cdots \cup Hb_t,
$$
is the union of $t\ge 2$ different cosets of $H$, where $b_1=1$.

Suppose that $t\ge 3$. Let $1,a,a'$ be three elements in $A$. 
By \eqref{eq:intperr}, we have
$$
|BA|\ge |B\cdot\{ 1,a,a'\}|\ge 3|B|-3|H|\ge 2|B|,
$$
which, by the same reasoning as in  \eqref{eq:lemk2s:3}, leads to a contradiction. Thus $t=2$. 
It follows that $|BA|=3|H|$ since otherwise we again have $|BA|\ge
2|B|$. 

Suppose that $|A|\geq 4$. Let $1,a,a',a''$ be four points in $A$. 
The four right translates $B, Ba, Ba', Ba''$ must each consist of
two right cosets of $H$ chosen among the three right $H$-cosets of
$BA$. But  by \eqref{eq:intperr} they must be distinct: this is not
possible, so we are left with the case $|A| = 3$.

In this case \eqref{eq:lemk2s:1} implies $|H|= |BA|-|B|\le 2|A|-4=2$ and $|BA|=6$. 

Let us set $H=\{1,h\}$, $B=H\cup Hb$, $BA=H\cup Hb\cup Bc$ and $A=\{1,a,a'\}$.
By \eqref{eq:intperr} we have, without loss of generality,
\begin{align}
  \label{eq:Ba}
  Ba &= H\cup Hc\\
  Ba'&= Hb\cup Hc\label{eq:Ba'}
\end{align}
The intersection property applied to $A$ implies that $a\not\in H$,
and $Ha\neq Ha'$,
hence \eqref{eq:Ba} implies $Ha=Hc$ and $Hba=H$ i.e. $Hb=Ha^{-1}$,
and \eqref{eq:Ba'} implies $Ha'=Ha^{-1}$ from which we have
$a'=ha^{-1}$
otherwise $A$ is a geometric progression and cannot satisfy the
intersection property. Equality \eqref{eq:Ba'} also implies
$Ha^{-1}a'=Ha$ from which we get $a'=aha$ since we cannot have
$a'=a^2$ because $A$ would again be a geometric progression.
From this and $a'=ha^{-1}$  we get 
$a^{-1}=haha$ from which $a^3=1$ and $ah=ha$. 
Therefore $B$ and $A$ together generate a
group of order $6$ inside which the $6$ left translates of $A$ must be
distinct and intersect in at most $1$ element. This is not possible.
\end{proof}

We are now ready to prove our upper bound on the size of aperiodic $2$--atoms.

\begin{proposition}\label{prop:m+3} Let $A$ be a $2$--atom of $S$ with
  $|A|\ge 3$. 
If $A$ is not left--periodic then $|A|\le \kappa_2 (S)-|S|+3$.
\end{proposition}

\begin{proof} Set $m=\kappa_2 (S)-|S|$. We may assume that $1\in A$.
Moreover $\kappa_2 (S)\ge |S|-1$ since otherwise $\kappa_1
(S)=\kappa_2 (S)<|S|-1$ 
and $A$ is a subgroup. By Lemma \ref{lem:sidon}, $A$ is $2$--separable and $\kappa_2 (A^{-1})=\kappa_2 (A)=2|A|-3$.

Suppose first that $A$ generates the same group $G$ as $S$. Then, $S$ is a witness that $\kappa_2 (A^{-1})\le |A|+m$, since $|S^{-1}A^{-1}|-|S^{-1}|=|AS|-|S|\le |A|+m$. By combining the two inequalities for $\kappa_2 (A)$ we get $|A|\le m+3$ as claimed.

Suppose now that $A$ generates a proper subgroup $H$ of $G$. 
Let $S=S_1\cup \cdots \cup S_k$ be the right--decomposition of $S$ modulo $H$, namely, each $S_i$ is the nonempty intersection of $S$ with a right--coset of $H$.

If there is an $S_i$ such that $|S_i|\ge 2$ and $|AS_i|\le |H|-2$, say $i=1$, then
$$
|A|+|S|+m=|AS|=\sum_i |AS_i|\ge |AS_1|+|S\setminus S_1|,
$$
so that $\kappa_2 (A)\le |A|+m$ and the above argument applies.

Suppose that  $|S_i|=1$ or $|AS_i|\ge |H|-1$ for each $i$.   Let $t$ be the number of $i$'s such that  $|AS_i|\le  |H|-1$. By Lemma \ref{lem:sidon}, we have $\kappa_1 (A)=|A|-1$, which gives
\begin{equation}\label{eq:m+3a}
|A|+|S|+m=|AS|=\sum_i |AS_i|\ge |S|+t(|A|-1),
\end{equation}
We cannot have $t=0$ since otherwise $AS=HS$ and $|HS|-|H|<|AS|-|A|$, contradicting that $A$ is a $2$--atom. If $t=1$, then
either $|S_1|=1$ or $|AS_1|=|H|-1$. 
In the first case, by taking $X=\{1,x\}\subset H$ we have
$|XS|-|X|\leq |AS|-|A|$, contradicting again that $A$ is a $2$--atom. 
In the second case we have $|AS|=|HS|-1$ and, since $|AS|-|A|\le |HS|-|H|$, we get $|A|\ge |H|-1$, which is incompatible with the intersection property of $A$. Hence $t\ge 2$ and \eqref{eq:m+3a} gives $|A|\le m+2$.
\end{proof}

\section{Periodic $2$--atoms}\label{sec:periodic}

Throughout the section $G$ is a finite group and $S$ is a
$2$--separable generating set of $G$ with $1\in S$ and $|S|\ge 3$. 

In this section we show that, if the $2$--atom of $S$ is left--periodic and $\kappa_2 (S)\le |S|-1$, then either there is a subgroup which is a $2$--fragment or the $2$--atom is the union of at most two cosets of a subgroup.  This fact is already stated in the preprint of Hamidoune \cite[Theorem 8.1]{yh09} when $\kappa_2 (S)=|S|-1$, deduced from a more general result in the setting of vertex transitive gaphs. We trace back the argument and give a direct proof which is slightly simpler.

If   $\kappa_1 (S)<|S|-1$ then $\kappa_2 (S)=\kappa_1 (S)$ and, by Theorem \ref{thm:1atom}, the atom of $S$ containing $1$ is a subgroup. We thus may assume that $\kappa_2 (S)=\kappa_1 (S)=|S|-1$.

Following Hamidoune we shall use the following diagram which is useful in the coming arguments. Let ${\cal F}=\{F_i,\; i\in I\}$ be a collection of $2$--fragments of $S$.  Each $F_i$ induces the partition $\{F_i,\partial F_i, G\setminus F_iS\}$ of $G$, where $F_i^*=G\setminus F_iS$ is a $2$--fragment of $S^{-1}$ and $\partial^{-1} F_i^*=\partial F_i$. For every pair $F_i, F_j\in {\cal F}$ we consider the common refinement of  these partitions and use the notation
$$
\beta_{ij}=|F_i\cap \partial  F_j| \mbox{  and } \beta'_{ij}=|\partial F_i\cap F_j^*|,\; i\neq j,
$$
and
$$
\gamma_{ij}=\gamma_{ji}=|\partial F_i\cap \partial F_j|,
$$
which is illustrated in Figure \ref{fig:diag}.

\begin{figure}[htbp]
\begin{center}
\setlength{\unitlength}{5mm}
\begin{picture}(12,10)
\put(2,1){\line(1,0){8}}
\put(2,4){\line(1,0){8}}
\put(2,6){\line(1,0){8}}
\put(2,9){\line(1,0){8}}

\put(2,1){\line(0,1){8}}
\put(5,1){\line(0,1){8}}
\put(7,1){\line(0,1){8}}
\put(10,1){\line(0,1){8}}

\put(2,0){$F_i$}
\put(4.5,-0.5){{\small $\begin{array}{ll}\partial F_i=\\\partial^{-1}F_i^*\end{array}$}}
\put(9,0){$F_i^*$}

\put(10.5,2){$F_j$}
\put(10.5,4.5){{\small $\partial F_j=\partial^{-1}F_j^*$}}
\put(10.5,7){$F_j^*$}

\put(3,4.5){$\beta_{ij}$}
\put(5.5,7){$\beta_{ij}'$}

\put(5.5,4.5){$\gamma_{ij}$}

\put(5.5,2){$\beta_{ji}$}
\put(8,4.5){$\beta_{ji}'$}
\end{picture}
\caption{The partition induced by  $(F_i, \partial F_i, F_i^*)$ and $(F_j, \partial F_j, F_j^*)$.}
\label{fig:diag}
\end{center}
\end{figure}
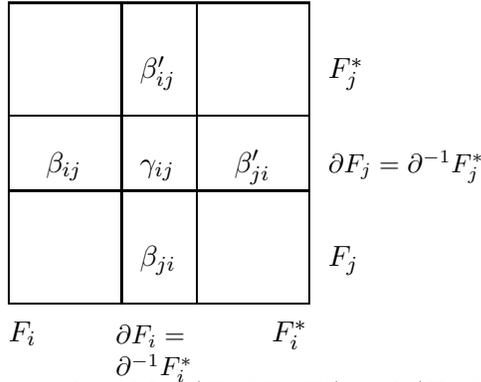

We start with a technical lemma which is a simple variant of the intersection
property in the case when $\kappa_2(S)=\kappa_1(S)$.

\begin{lemma}\label{lem:beta}
  Suppose $\kappa_2(S)=\kappa_1(S)$. Let $F_1$ and $F_2$ be two
  $2$-fragments of $S$. Suppose $F_1\cap F_2\neq\emptyset$.
  Then with the notation of Figure~\ref{fig:diag} we have
  $$\beta_{12}\geq \beta_{12}'.$$
  Suppose furthermore that $F_1^*\cap F_2^*\neq\emptyset$. Then
  $\beta_{12}=\beta_{12}'$, $\beta_{21}=\beta_{21}'$ and
  \begin{enumerate}[(i)]
  \item either $F_1\cap F_2$ is a $2$-fragment of $S$,
  \item or $|F_1\cap F_2|= 1$.
  \end{enumerate}
\end{lemma}

\begin{proof}
  We have 
  \begin{equation}
    \label{eq:partial}
\beta_{12}+\gamma_{12}+\beta_{21}\ge |\partial (F_1\cap F_2)|\ge \kappa_1=\kappa_2=\beta'_{12}+\gamma_{12}+\beta_{21},
  \end{equation}
  which implies $\beta_{12}\geq \beta_{12}'.$
  If both $F_1\cap F_2$ and $F_1^*\cap F_2^*$ are non empty, then
  summing \eqref{eq:partial} together with
  $$
  \beta_{12}'+\gamma_{12}+\beta_{21}'\ge |\partial (F_1^*\cap F_2^*)|\ge \kappa_1=\beta_{12}+\gamma_{12}+\beta_{21}'
  $$
  we obtain that all inequalities must in fact be equalities, meaning
  that
  $$|\partial (F_1\cap F_2)| = |\partial (F_1^*\cap F_2^*)| = \kappa_1
  = \kappa_2
  $$
  which implies the result.
 \end{proof}

We consider two cases according to whether or not 
$|G|<2\alpha_2(S^\varepsilon)+\kappa_2(S^\varepsilon)$
holds for $\varepsilon=1$ or $\varepsilon=-1$.

\begin{lemma}\label{lem:k2=k1} 
Suppose $\kappa_2(S)=\kappa_1(S)$ and either 
$|G|<2\alpha_2(S)+\kappa_2(S)$ or
$|G|<2\alpha_2(S^{-1})+\kappa_2(S^{-1})$.
Then there is a $2$--fragment of $S$ or a $2$--fragment of $S^{-1}$ 
which is a subgroup.
\end{lemma}

\begin{proof}  
Let us suppose $|G|<2\alpha_2(S^{-1})+\kappa_2(S^{-1})$, the case
$|G|<2\alpha_2(S)+\kappa_2(S)$ being similar.
Let $F_1, F_2$ be any $2$--fragments of $S$ such that $F_1^*$ and $F_2^*$ are $2$--atoms of $S^{-1}$. We shall show that $F_1\cap F_2=\emptyset$. It will follow that one such $2$--fragment containing $1$ satisfies $x^{-1}F=F$ for each $x\in F$ and thus $F$ is a subgroup, as claimed.

Suppose on the contrary that $F_1\cap F_2\neq\emptyset$. 
By Lemma~\ref{lem:beta} we have $\beta_{12}\ge \beta'_{12}$. 

Now the hypothesis $|G|<2\alpha_2(S^{-1})+\kappa_2(S^{-1})$ translates
into $|F^*_1|>|F_1|$ since $F^*_1$ is a $2$-atom of $S^{-1}$ (see the
remark after Theorem~\ref{thm:int}).

We have (see Figure~\ref{fig:diag}) 
$$|F_1|= |F_1\cap F_2| + \beta_{12} + |F_1\cap F_2^*|$$
and
$$|F_1^*| = |F_2^*| = |F_1^*\cap F_2^*| + \beta_{12}' + |F_1\cap
F_2^*|$$
from which $|F^*_1|>|F_1|$ implies, since $\beta_{12}\geq\beta_{12}'$,
$$|F^*_1\cap F^*_2|>|F_1\cap F_2|>0.$$
Since  $F_1^*$  is a $2$--atom of $S^{-1}$, it follows that
$$
\beta_{12}'+\gamma_{12}+\beta_{21}'\ge |\partial (F_1^*\cap F_2^*)|>\kappa_2 (S)=\beta_{12}+\gamma_{12}+\beta'_{12},
$$
which implies $\beta'_{12}>\beta_{12}$, a contradiction.
\end{proof}

According to the last lemma we can restrict ourselves to the case when we
simultaneously have
$|G|\ge 2\alpha_k(S)+\kappa_k(S)$ and $|G|\ge 2\alpha_k(S^{-1})+\kappa_k(S^{-1})$.

\begin{lemma}\label{lem:2coset} 
Assume that $\kappa_2 (S)=\kappa_1(S)=|S|-1$ and $\alpha_2 (S)>2$ and $\alpha_2 (S^{-1})>2$. 
If $|G|\ge 2\alpha_k(S)+\kappa_k(S)$ and $|G|\ge
2\alpha_k(S^{-1})+\kappa_k(S^{-1})$,
then there is a $2$--atom of $S$ or of $S^{-1}$ which is  the union of
at most two right cosets of $H$.
\end{lemma}

\begin{proof} 
Let $A$ be a $2$--atom of $S$ with $1\in A$.  Since $\alpha_2 (S)>2$,
by Proposition \ref{prop:m+3}, $A$ is left--periodic. Suppose that $A$
is the union of at least three cosets of a subgroup $H$. We may assume
$H$ to be the stabilizer of $A$ 
by left translations. Thus there are $x,y\in A$ such that $A_1=A,
A_2=x^{-1}A, 
A_3=y^{-1}A$ are pairwise distinct and contain $1$. By the intersection property, $A_1\cap A_2\cap A_3=\{ 1\}$.

Applying the first part of Lemma~\ref{lem:beta} with $F_i=A_i$ we
obtain that $|A_1\cap A_2|=1$ implies
\begin{equation}\label{eq:2coset1}
|A_1^*\setminus A_2^*|\le |A_1|-1.
\end{equation}
We observe that $A_1^*\setminus A_2^*\neq \emptyset$ since otherwise
$A_1^*=A_2^*$ implies $A_1=A_2$. 
Now the hypothesis $|G|\ge 2\alpha_k(S)+\kappa_k(S)$ translates to
$|A_1|\leq |A_1^*|$ so that \eqref{eq:2coset1} implies $A_1^*\cap
A_2^*\neq\emptyset$.
The second part of Lemma~\ref{lem:beta} therefore applies and
decomposing $A_1S\cap A_2S$ as
  $$A_1S\cap A_2S = (A_1\cap A_2) \cup (A_1\cap \partial A_2) \cup
  (A_2\cap \partial A_1) \cup (\partial A_1\cap \partial A_2)$$
we have 
\begin{equation}\label{eq:2coset2}
|A_1S\cap A_2S|=\kappa_1 +1\le |S|.
\end{equation}

Let $Z$ be a $2$--atom of
$S^{-1}$. By analogy to the previous case, we may assume that $Z$ is
left--periodic and the union of at least three cosets of its
stabilizer by left translations (otherwise we are done.) Thus there
are three distinct $2$--atoms $Z_1, Z_2, Z_3$ of $S^{-1}$ with an only
common point $z$. By translating $Z_1, Z_2, Z_3$ if need be, we impose $z\in A_1^*\setminus A_2^*$.

By exactly the same argument with $S^{-1}$ replacing
$S$ we deduce
\begin{equation}\label{eq:2coset2w}
|Z_1S^{-1}\cap Z_2S^{-1}|\le |S|.
\end{equation}

Let us now suppose, without loss of generality, that $\alpha_2(S)\leq \alpha_2(S^{-1})$.
If this is not the case, simply switch $S$ and $S^{-1}$.

Suppose first that $A_1^*$ contains $Z_1\cup Z_2$. This implies
$Z_1^*\cap Z_2^*\supset A_1$, in particular $1\in Z_1^*\cap Z_2^*$.
Therefore we also have $Z_i^*\cap A_2\neq \emptyset$ for $i=1,2$, so
that Lemma~\ref{lem:beta} implies that $Z_i\cap A_2^*$ is a
$2$-fragment of $S^{-1}$ or of cardinality $1$. But we have chosen $z\in Z_i$ and
$z\not\in A_2^*$, so $Z_i\cap A_2^*$ is strictly included in $Z_i$,
which means that $Z_i\cap A_2^*$ cannot be a $2$--fragment of
$S^{-1}$, because $Z_i$ is a $2$--atom.
Therefore $|Z_i\cap A_2^*|\le 1$ for $i=1,2$. But then, using
$|A_1|\leq |Z_1|$ because we have supposed $\alpha_2(S)\leq \alpha_2(S^{-1})$,
$$
|(Z_1\cup Z_2)\cap A_1^*\setminus A_2^*|\ge 2|Z_1|-3\ge 2|A_1|-3>|A_1|-1,
$$
contradicting \eqref{eq:2coset1}. 

We may thus suppose that $A_1^*$ contains at most one of the three
atoms $Z_1, Z_2, Z_3$. without loss of generality we may assume that
each of $Z_1$ and $Z_2$ are not contained in $A_1^*$. 

Now $Z_i\not\subset A_1^*$ implies $A_1\not\subset Z_i^*$. $A_1$ is a
$2$-atom of $S$, $Z_i$ is a $2$-fragment of $S$, and
since $\alpha_2 (S)\le \alpha_2 (S^{-1})$, the intersection property
of Theorem~\ref{thm:intAF} implies $|A_1\cap Z_i^*|\le 1$.

This implies that $A_1$ intersects $Z_1S^{-1}\cap Z_2S^{-1}$ in at least $|A_1|-2$ points. Moreover, since $z\in A_1^*$, we have $zS^{-1}\cap A_1=\emptyset$. It follows that, by using \eqref{eq:2coset2w},
$$
|S| \ge |Z_1S^{-1}\cap Z_2S^{-1}|
\ge |zS^{-1}|+|Z_1|-2\ge |S|+|A_1|-2,
$$
contradicting that $|A_1|=\alpha_2 (S)>2$. This completes the proof.
\end{proof}
%}

\section{Periodic $2$--atoms which are not subgroups}
\label{sec:notsubgroups}

Throughout the section $G$ is a finite group and $S$ is a
$2$--separable generating set of $G$ with $1\in S$,
$|S|\ge 3$, and $\kappa_2(S)=\kappa_1(S)=|S|-1$.
We assume that no $2$-atom of $S$ or of $S^{-1}$ is a subgroup or has
cardinality $2$. By Lemmas~\ref{lem:k2=k1} and \ref{lem:2coset} either
$S$ or $S^{-1}$ has a $2$-atom 
which is the union of two right cosets of some subgroup $H$,
$$
A=H\cup Ha.
$$

Our main goal in this section is to prove the following result.

\begin{theorem}\label{thm:2coset}
  Let $S$ be such that $|G|\geq 2\alpha_2(S) +\kappa_2(S)$ and
  $\kappa_2(S)=\kappa_1(S)=|S|-1$. If $S$ does not have $2$--fragments that
  are subgroups and
  if $A=H\cup Ha$ is a $2$-atom of $S$, then $HS$ consists of the
  complement of exactly two right-cosets modulo $H$.
\end{theorem}

{\bf Remark.} The complement of $HS$ must contain at least two cosets
modulo $H$, otherwise we have $|G|=\alpha_2(S)+\kappa_2(S)+|H| <
2\alpha_2(S) +\kappa_2(S)$.

\subsection{Reduction to the case when $|H|<6$}

Our approach will be the following: {\em left} multiplication by $A$
defines a graph on the set of right-cosets $Hx$ modulo the subgroup
$H$. This graph (to be defined precisely below) is arc-transitive. Now
the study of arc-transitive graphs shows them to have generally high
connectivity. On the other hand, the condition
$\kappa_2(S)=\kappa_1(S)=|S|-1$ implies relatively low connectivity
for the arc-transitive graph, and we will obtain a contradiction for
all cases but the one mentioned in Theorem~\ref{thm:2coset}.

We first observe that $a^{-1}A\neq A$, since otherwise $A$ is
left--periodic by the subgroup generated by $H\cup \{ a\}$ and
therefore $A$ coincides with this subgroup. Hence, since $1\in A\cap
a^{-1}A$, we have, by the intersection property for $2$-atoms (Theorem~\ref{thm:int}):
\begin{equation}\label{eq:conj}
H\cap a^{-1}Ha=\{ 1\}.
\end{equation}

We also observe that
\begin{equation}\label{eq:hs}
HS=AS,
\end{equation}
since otherwise there is a full right--coset $Hx$ contained in  $AS\setminus HS$ and
$$
|HS|-|H|\le |AS|-|H|-|Hx|=|AS|-|A|,
$$
which contradicts $A$ being a $2$--atom.

Let $X=Cay(G,A)$ be the {\em left} Cayley graph of $A$   (arcs are
$x\to \alpha x,\; \alpha\in A$). We define the (right) quotient graph
$X/H$ which has  vertex set 
$\{ Hx, x\in G\}$, the right cosets of $H$, and there is an arc $Hx\to
Hy$ if and only if $HaHx \supset Hy$.

\begin{lemma}\label{ref:X/H} 
The graph $X/H$ is vertex--transitive and arc--transitive. 
In particular $X/H$ is regular and its degree is $|H|$.
\end{lemma}

\begin{proof} 
For each $z\in G$ we have  $HaHx\supset Hy$ if and only if
$HaHxz\supset Hyz$, 
so that the right translations $Hx\to Hxz$ are automorphims of $X/H$
and the set of right translations acts 
transitively on the vertex set of $X/H$.

Being vertex--transitive, to show that $X/H$ is also arc--transitive,
it suffices to show that the subset of $Aut(X/H)$ that leaves $H$ invariant acts
transitively on the set of neighbours of $H$. 
This follows by choosing left translations by $z\in H$.

In particular $X/H$ is a regular graph. Observe that by
\eqref{eq:conj}, there are $|H|$ distinct right cosets in $HaH$. Therefore, $X/H$ is regular of degree $|H|$.
\end{proof}

The key observation in the use of the graph $X/H$ is the following
one. 
According to \eqref{eq:hs} and the fact that $\kappa_2 (S)=|S|-1$, we have
$$
|HS|-|S|=|AS|-|S|=|A|-1=2|H|-1.
$$
By looking at $HS$ in the graph $X/H$ we see that all arcs emanating
from $HS$ to $G\setminus HS$ lead to cosets $Has$ for some $s\in HS\setminus S$.
Hence, denoting by $e(HS)$ the number of arcs leading out from $HS$,

\begin{equation}\label{eq:2h-1}
e(HS)\le 2|H|-1.
\end{equation}

At this point we will use some properties of arc--connectivity in
arc--transitive graphs which can be found in \cite{hlst}.  The theory
of atoms can be formulated for the arc--connectivity of graphs, in
which setting it is somewhat simpler. For a subset $C$ of the vertex
set $V(Y)$ of a connected graph $Y$, we denote by $e(C)$ the set of arcs
connecting a point in $C$ to a point in $V(Y)\setminus C$. 
If $k\le |V(Y)|/2$, we shall say that a subset of vertices $C$ is
{\em $k$--separating} if it has cardinality at least $k$ and the set
of vertices not in $C$ has cardinality at
least $k$. We shall say that the graph $Y$ is $k$-separable is there
exists a $k$--separating set.
the $k$--arc connectivity $\lambda_k(Y) $ of $Y$ is the minimum number of arcs leading out
of a $k$--separating set, in other words:
$$
\lambda_k(Y)=\min\{ |e(C)|: k\le |C|\le |V(Y)|-k\}.
$$
An arc $k$--fragment of $Y$ is a set $F$ of vertices with
$e(F)=\lambda_k(Y)$, and an arc $k$--atom of $Y$ is an arc
$k$--fragment with minimum cardinality.

The next Lemma is Corollary 5 from \cite{hlst}.

\begin{lemma}\label{lem:karc} 
Let $Y$ be a connected arc--transitive graph with (out)degree $d$. 
If $Y$ is $k$-separable then the arc $k$-atoms have cardinality at most
$2k-2$.
If furthermore 
$k\le d/3+1$, then
every arc $k$--atom of $Y$ has cardinality $k$. In particular,
$$
\lambda_k (Y)\ge dk-e_k(Y),
$$
where $e_k(Y)$ is the largest number of arcs in a subgraph induced by
a set of cardinality $k$. Moreover, the same conclusion holds for $k$
up to $2d/3+1$ 
if $G$ is antisymmetric (i.e. it has no $2$--cycles).
\end{lemma}

We will apply Lemma~\ref{lem:karc} to obtain a contradiction with the
hypothesis of Theorem~\ref{thm:2coset} in the case when $HS$ is a
$3$-separating set of $X/H$. This will yield the conclusion. We first
show that we can limit ourselves to studying the case when $X/H$ is a connected graph.

Note that $X/H$ is not connected if and only if $\subgp{A}$ is a
proper subgroup of $G$. Consider the partition $S=S_1\cup S_2\cup
\ldots \cup S_m$ where $S_i$ is the non-empty intersection of $S$ with
some right coset of $G$ modulo $\subgp{A}$. Since 
$S$ generates $G$, we have $m>1$ if $X/H$ is not connected.

\begin{claim}\label{claim:S_i}
  We have $HS_i = \subgp{A}S_i$ for all $i$, $1\leq i\leq m$, except
  for one
  value of $i$.
\end{claim}

\begin{proof}
  Every subgraph of $X/H$ induced by $HS_i$ is connected and
  arc-transitive, hence it has $\lambda_1 = |H|$ by Lemma~\ref{lem:karc}
  since the degree of $X/H$ is $|H|$. Therefore if $HS_i$ is
  $1$--separating in its connected component, it has at least $|H|$
  outgoing edges. But \eqref{eq:2h-1} implies that there can only be a
  single such $HS_i$. Note that we cannot have $HS_i = \subgp{A}S_i$
  for all $i$ otherwise $AS=\subgp{A}S$ which contradicts $A$ being a
  $2$-atom of $S$.
\end{proof}

  Without loss of generality, let $S_1$ be such that
  $HS_1\neq\subgp{A}S_1$ and $1\in S_1$.

  \begin{claim}\label{claim:2|H|}
    We have $|AS_1|\leq |\subgp{A}|-2|H|$. 
  \end{claim}

  \begin{proof}
    If not, then $|AS_1| = |\subgp{A}|-|H|$. But either $\subgp{A}S
    \neq G$, and $|\subgp{A}S| - |\subgp{A}|\leq |AS|-|A|$ which
    contradicts the hypothesis of Theorem~\ref{thm:2coset}
    that no fragment of $S$ is a subgroup, or $\subgp{A}S =G$, but
    then $|AS|=|G|-|H|$ so that $|G|=|AS|+|H|<|AS|+|A|$ meaning
    $|G|<2\alpha_2(S)+\kappa_2(S)$ which also contradicts the
    hypothesis of Theorem~\ref{thm:2coset}.
  \end{proof}

  \begin{claim}
    For every $i>1$, we have $S_i=\subgp{A}S_i$.
  \end{claim}

  \begin{proof}
    If $S_i\neq\subgp{A}S_i$ for some $i\neq 1$, then
    $|\subgp{A}(S\setminus S_1)| > |S\setminus S_1|$ 
    and Claim~\ref{claim:S_i} implies
    \begin{equation}
      \label{eq:AS-A}
      |S|-1=|AS|-|A|=|AS_1|-|A| + |\subgp{A}(S\setminus S_1)|
    \end{equation}
     so that we have
    $$|AS_1|-|A| < |S_1|-1.$$
    In other words, $\kappa_1(S_1)<|S_1|-1$.
    Furthermore, if $B$ is a $2$-fragment of $S_1$, then we have
    \begin{align*}
      |BS|-|B| &= |BS_1|-|B| + |BS| -|BS_1|\\
               &\leq |BS_1|-|B| + |\subgp{A}(S\setminus S_1)|
    \end{align*}
    and by writing, since $A$ is a $2$-fragment of $S$,
    $$|AS|-|A| \leq |BS|-|B|,$$
    we obtain that, applying \eqref{eq:AS-A} again,
    $$|AS_1|-|A| \leq |BS_1|-|B|$$
    so that $A$ is also a $2$-fragment of $S_1$. Note that since $B$
    is a $2$-fragment of $S_1$, the two preceding inequalities must be
    equalities, so that $B$ must also be a $2$-fragment of $S$.
    By Claim~\ref{claim:2|H|} we have $|AS_1|\leq |\subgp{A}|-|A|$,
    therefore $|\subgp{A}|\geq 2\alpha_2(S_1)+\kappa_2(S_1)$, meaning
    that the intersection property must hold for $1$-atoms of $S_1$,
    and there is a $1$-atom of $S_1$ that is a non-trivial subgroup. This subgroup
    must be a $2$-fragment of $S$, contradicting the hypothesis of
    Theorem~\ref{thm:2coset}.
  \end{proof}

Finally we can now state:

\begin{lemma}
  Under the hypothesis of Theorem~\ref{thm:2coset}, replacing $S$ by a
  right translate if need be, we have
  $$S = (S\cap\subgp{A}) \cup (G\setminus\subgp{A}).$$
\end{lemma}

\begin{proof}
  The last claim means that $\subgp{A}(S\setminus S_1) = S\setminus
  S_1$. By \eqref{eq:hs} we must have $|S_1|>1$.
  Now if $S\setminus S_1\neq G\setminus\subgp{A}$, then
  $\subgp{A}S\neq G$ and $|\subgp{A}S|-|\subgp{A}| < |S|-1$,
  contradicting the hypothesis of Theorem~\ref{thm:2coset}.
\end{proof}

The preceding lemma shows that $\kappa_1(S_1)=|S_1|-1$ and that any
$2$-atom of $S_1$ is a $2$-atom of $S$. We see therefore that it
suffices to
prove Theorem~\ref{thm:2coset} with the additional hypothesis that
$A$ generates $G$. We will henceforth suppose this to be the case, so
that we are dealing with a connected graph $X/H$.

We now work towards proving:
\begin{proposition}\label{prop:lambda3}
  Suppose $S$ has a $2$--atom $A=H\cup Ha$ which is the union of two
  right cosets of some subgroup $H$ and suppose that $A$ defines a connected
  graph $X/H$. Then the subset of vertices of $X/H$ defined by $HS$ is
  not $3$--separating.
\end{proposition}

{\bf Remark.} We have $|AS|-|A| = \kappa_2(S) >0$, therefore $|AS|>2|H|$
so that by \eqref{eq:hs} we have $|HS|>2|H|$: hence $HS$ in $X/H$
consists of at least three vertices. Therefore $HS$ defines a
$3$-separating set if and only if there are at least $3$ vertices of
$X/H$ in the complement of $HS$.

\begin{proof}[Proof of Proposition~\ref{prop:lambda3} when $|H|\geq
  6$]
  By Lemma \ref{lem:karc} with $k=3$, since the degree of $X/H$ is
  $|H|\geq 6$ and $e_3(X/H)\le 6$, we have 
\begin{equation}\label{eq:l3}
\lambda_3 (X/H)\ge 3|H|-6>2|H|-1,
\end{equation}
contradicting \eqref{eq:2h-1}, so that $HS$ cannot define a
$3$--separating subset of vertices of $X/H$.
\end{proof}

We next study the small values of $|H|$. 

\subsection{Proof of Proposition~\ref{prop:lambda3} in the case $|H|=3,4,5$}

First note that an arc-transitive graph is either symmetric, or
antisymmetric (meaning that if there is an arc from $x$ to $y$ then
there is no arc from $y$ to $x$). If $X/H$ is antisymmetric, then
$e_3(X/H)\leq 3$. If $|H|\geq 3$ then
Lemma~\ref{lem:karc} applies for $k=3$ and if $HS$ is $3$-separating in $X/H$ we have
$\lambda_3(X/H)\geq 3|H|-3\geq 2|H|$ which contradicts \eqref{eq:2h-1}.

We are therefore left with the case when $X/H$ is symmetric. We now
rule this out.

\begin{lemma}\label{lem:symmetric}
  If $A=H\cup Ha$ is a $2$-atom of $S$, then the graph $X/H$ cannot be symmetric.
\end{lemma}

To prove Lemma~\ref{lem:symmetric} we will need the following easy fact.

\begin{lemma}\label{lem:stable}
  When $a^2\in H$, then for any $x\in G$, the set $T=Hx\cup a^{-1}Hx$
  is stable by left multiplication by $a$, i.e. $aT=T$.
\end{lemma}

\begin{proof}
  We have $a^2=h$ for some $h\in H$, hence $a=a^{-1}h$, therefore $aHx=a^{-1}Hx$.
\end{proof}

\begin{proof}[Proof of Lemma~\ref{lem:symmetric}]
   First note
that, by changing $a$ to $ha$ for some $h\in H$ if need be, $X/H$ is symmetric can
be taken to mean that $a^2\in H$.

 Let $C$ be the complement of $HS$ in $G$. Since $HS=HaS$, no element of
  $a^{-1}C$ can be in $S$. 
  Now, by Lemma~\ref{lem:karc} the arc $2$--atoms of $X/H$ have cardinality
  $2$ and $\lambda_2(X/H) = 2|H|-2$. By the remark following Theorem~\ref{thm:2coset}, the set $HS$ must
  define a $2$-separating subset of vertices of $X/H$,
  therefore
  $a^{-1}C$ intersects at least $2|H|-2$ different $H$-cosets of
  $HS$. Since $|HS|=|AS|=|S| + 2|H|-1$, we have that the complement of $S$
  equals
   $$\overline{S} = C\cup a^{-1}C \cup E$$
  where $E$ is either the empty set or a single element. But since,
  by Lemma~\ref{lem:stable}, $a(C\cup a^{-1}C) = C\cup a^{-1}C$,
  we have that $\{1,a\}S = S$ (if $E=\emptyset$) or
  $|\{1,a\}S|=|S|+1$ (if $|E|=1$) so that $A$ cannot be a $2$-atom of
  $S$. 
\end{proof}

This concludes the proof of Proposition~\ref{prop:lambda3} in the case $|H|=3,4,5$.

\subsection{Proof of Proposition~\ref{prop:lambda3} in the case
  $|H|=2$}

We are now dealing with a graph $X/H$ of degree $|H|=2$ that must be
antisymmetric by Lemma~\ref{lem:symmetric}. We assume that $HS$
defines a $3$-separating set of $X/H$ and work towards a contradiction.
Inequality \eqref{eq:2h-1}
tells us that we must have $\lambda_3\leq 2|H|-1=3$. Non-trivial
arc-transitive graphs with these parameters do exist however and we do
not have a contradiction directly. We shall therefore consider also
arc $4$--connectivity. By pure graph-theoretic arguments we will
obtain that we must have $\lambda_4\geq 4$ which will mean that $HS$
cannot be $4$--separating in $X/H$ otherwise we would contradict
\eqref{eq:2h-1}.
We will then be left with the case when $HS$ is $3$--separating but
not $4$--separating, meaning that $HS$ consists of either three cosets
modulo $H$ or the complement of three cosets. We will then conclude
the proof of Proposition~\ref{prop:lambda3} by excluding these cases separately.

Notice that if $X/H$ contains
no triangles (with any orientation) then subsets of $3$ or $4$
vertices must have at least $4$ outgoing arcs, so that $\lambda_3\geq
4$ since arc $3$--atoms are of cardinality $3$ or $4$ (Lemma~\ref{lem:karc}). We may therefore
assume that every edge of $X/H$ is contained (by arc-transitivity) in
an oriented triangle:

    %\begin{center}
      
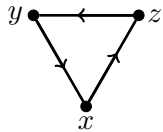
\begin{figure}[h]
        \centering
         \begin{tikzpicture}[scale=1.4]
         \tikzstyle{every node}=[draw,circle,fill=black,minimum size=4pt,
                            inner sep=0pt]
          \draw (120:1cm) node (y) [label=left:$y$] {} -- (0,0) node (x)
          [label=below:$x$]{} -- (60:1cm) node (z) [label=right:$z$]
          {};
          \begin{scope}[line width=1pt]
          \draw (z) -- (y) -- (x) -- (z);
          \draw[-to, shorten >=0.5cm] (z) -- (y);
          \draw[-to, shorten >=0.5cm] (y) -- (x);
          \draw[-to, shorten >=0.5cm] (x) -- (z);
          \end{scope}
       \end{tikzpicture}
      \end{figure}
      
   %  \end{center}

Indeed, both the in- and the out-neighbourhood of a vertex are
vertex-transitive digraphs on two vertices, such a neighbourhood
cannot therefore
contain an edge otherwise it would contain also the reverse edge.
Summarizing:

\begin{lemma}\label{lem:triangle}
 If $\lambda_3(X/H)\leq 3$ then every edge of $X/H$ belongs to an
oriented triangle. Furthermore arc $3$-atoms are of cardinality $3$
and are triangles.
\end{lemma}

Let us denote call a $K_4^*$ an antisymmetric graph on $4$ vertices
with $5$ arcs (a $K_4$ with an edge removed).

\begin{lemma}\label{lem:K4*}
  If $X/H$ is not an octahedron when orientation is removed, then it contains no $K_4^*$.
 \end{lemma}

\begin{proof}
Suppose the $4$-vertex set $\{a,b,c,d\}$ forms a $K_4^*$, then
orientations must be as below since triangles must be oriented.

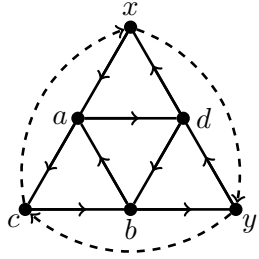
\begin{figure}[h]
  \centering
  \begin{tikzpicture}[scale=1.4,line width=1pt]
    \tikzstyle{every node}=[draw,circle,fill=black,minimum size=4pt,
                            inner sep=0pt]
     \draw (0,0) node (c) [label= below left:$c$] {} -- (1,0) node (b)
     [label= below:$b$] {} -- (2,0) node (y) [label= below right:$y$]
     {} -- ++(120:1cm) node (d) [label= right:$d$] {} -- ++(120:1cm)
     node (x) [label= above:$x$] {} -- ++(240:1cm) node (a) [label=
     left:$a$] {} -- (c);
     \draw (a) -- (b) -- (d) -- (a);
     \draw[style=dashed,-to] (x) to [out=-20,in=80]   (y);
     \draw[style=dashed,-to] (y) to [out=-140,in=-40] (c);
     \draw[style=dashed,-to] (c) to [out=100,in=-160] (x);
     \begin{scope}[draw, -to, shorten >=0.5cm]
       \draw (c) -- (b); \draw (b) -- (a); \draw (a) -- (c);
       \draw (d) -- (x); \draw (x) -- (a); \draw (a) -- (d);
       \draw (b) -- (y); \draw (y) -- (d); \draw (d) -- (b);
     \end{scope}
  \end{tikzpicture}
\end{figure}

Now we must have an arc $(x,a)$ and an arc $(b,y)$ with $x,y$ outside
the $K_4^*$. Note that we must have $x\neq y$ otherwise the outer
neighbourhood of $b$ would not be of degree $0$. Note further that
every arc, like arc $(b,a)$, must belong to two distinct triangles.
For arc $(x,a)$, the two triangles can only be $(x,a,c)$ and
$(x,a,d)$,
therefore we must have arcs $(c,x)$ and $(d,x)$. Similarly, we must
have arcs $(y,c)$ and $(y,d)$, as in the picture. But then the picture
can only be completed by arc $(x,y)$ and we have an octahedron.
\end{proof}

We now claim:

\begin{proposition}\label{prop:lambda4}
  If $X/H$ is $4$-separable, then $\lambda_4\geq 4$.
\end{proposition}

Since $4$-atoms have cardinality $4,5$ or $6$ by Lemma~\ref{lem:karc}, we prove
Proposition~\ref{prop:lambda4} by showing that subsets of $4,5$ or $6$
vertices have at least $4$ outgoing arcs. We do this by looking at the
non-oriented version of the graph, obtained from $X/H$ by forgetting
orientation, and showing that any subset of $4,5,6$ vertices must have
at least $8$ outgoing edges. Note that we may suppose that the graph
contains no $K_4^*$, since octahedrons are not $4$-separable. Since we
may also assume that every edge is included in a triangle
(Lemma~\ref{lem:triangle}) we can claim:

\begin{lemma}\label{lem:01}
  If $X/H$ is $4$-separable and contains triangles, then the
  neighbourhood $N(x)$ of any vertex $x$ is a $4$-vertex graph of
  degree $1$.
\end{lemma}

From the fact that $X/H$ contains no $K_4^*$ we obtain immediately:

\begin{lemma}\label{lem:F=4}
  If $F$ is a $4$-separating set of $X/H$ with $|F|=4$,
  then $F$ has at least $8$ outgoing edges.
\end{lemma}

We now claim:

\begin{lemma}\label{lem:F=5}
 If $F$ is a $4$-separating set of $X/H$ with $|F|=5$,
  then $F$ has at least $8$ outgoing edges.
\end{lemma}

\begin{proof}
  Otherwise $F$ contains at least $7$ edges.
  If $F$ contains a vertex $x$ of degree $4$ then the remaining $3$ edges
  of $F$ must be in $N(x)$. But $N(x)$ contains at most $2$ edges by
  Lemma~\ref{lem:01}. Therefore $F$ cannot contain a vertex of degree $4$ and
  since it contains $7$ edges it must contain
  a vertex $x$ of degree~$3$. Let $y$ be the vertex of $F$ that is not
  in $N(x)$.

  \begin{figure}[h]
    \centering
    \begin{tikzpicture}
    \tikzstyle{every node}=[draw,circle,fill=black,minimum size=4pt,
                            inner sep=0pt]

     \draw (0,0) node (x) [label=left:$x$] {}
        -- (0,1) node (u) [label=left:$u$] {}
        -- (1,0) node (v) [label=right:$v$] {}
        -- (x)
        -- (1,1) node (w) {};
     \node (y) at (1.7,1.7) [label=right:$y$] {};

     \draw[style=dashed] (u) -- (y) -- (v) (w) -- (y);

  \end{tikzpicture}
  \end{figure}

There must be an edge between some two neighbours, say $u$ and $v$, of
$x$ in $F$, otherwise we can put at most six edges in $F$. Since there
cannot be any other edge in $N(x)\cap F$, the vertex $y$ must be
connected to all $3$ neighbours of $x$ in $F$. But then $x,u,v,y$ make
up a $K_4^*$, contradicting Lemma~\ref{lem:K4*}.
\end{proof}

\begin{lemma}\label{lem:F=6}
 If $F$ is a $4$-separating set of $X/H$ with $|F|=6$,
  then $F$ has at least $8$ outgoing edges.
\end{lemma}

\begin{proof}
  Otherwise $F$ contains at least $9$ edges.
  Suppose first that $F$ contains a vertex $x$ of degree $4$ in $F$.
  If there are no edges in $N(x)$ then $F$ clearly cannot contain $9$
  edges, therefore the neighbourhood of $x$ is of degree $1$ and has
  the structure below:

  
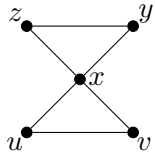
\begin{figure}[h]
    \centering
    \begin{tikzpicture}
    \tikzstyle{every node}=[draw,circle,fill=black,minimum size=4pt,
                            inner sep=0pt]

     \node (x) at (0,0) [label=right:$x$] {};

     \node (y) at (45:1cm)  [label=above right:$y$] {};
     \node (z) at (135:1cm) [label=above left:$z$] {};
     \node (u) at (225:1cm) [label=below left:$u$] {};
     \node (v) at (315:1cm) [label=below right:$v$] {};

     \draw (x) -- (y) -- (z) -- (x);
     \draw (x) -- (u) -- (v) -- (x);

  \end{tikzpicture}
  \end{figure}

But then, to fit three extra edges in $F$ from the remaining vertex $w$ of
$F$ that is neither $x$ nor in $N(x)$, we must connect $w$ to either
$x$ and $y$ or to $u$ and $v$, which creates a $K_4^*$, contradicting
Lemma~\ref{lem:K4*}.

Therefore, if $F$ contains nine edges, the only possibility left is
for $F$ to be regular of degree $3$. There are only two non-isomorphic
graphs of degree $3$ on six vertices which are:

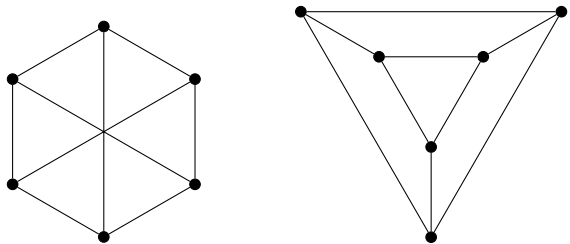
\begin{figure}[h]
  \centering
  \begin{tikzpicture}[scale=1.4]
    \tikzstyle{every node}=[draw,circle,fill=black,minimum size=4pt,
                            inner sep=0pt]

     \foreach \x in {0,1,...,5}
    {
        \node (\x) at (30+60*\x:1cm) {};
    };

    \draw (0) -- (1) -- (2) -- (3) -- (4) -- (5) -- (0);
    \draw (0) -- (3) (1) -- (4) (2) -- (5);
  \end{tikzpicture}
\hspace{1cm}
  \begin{tikzpicture}%[scale=1.2]
  \tikzstyle{every node}=[draw,circle,fill=black,minimum size=4pt,
                            inner sep=0pt]

   \foreach \x in {0,1,2}
    {
        \node (\x) at (30+120*\x:0.8cm) {};
    };
    \foreach \x in {3,4,5}
    {
        \node (\x) at (30+120*\x:2cm) {};
    };
    \draw (0) -- (1) -- (2) -- (0);
    \draw (3) -- (4) -- (5) -- (3);
    \draw (0) -- (3) (1) -- (4) (2) -- (5);
  \end{tikzpicture}
\end{figure}

The first graph contains no induced triangles: but we have seen that
every edge of $X/H$
must be included in some triangle. This implies in the case of
the first graph that every one of
its edges has both its endpoints connected to a common vertex. But
since the degree of $X/H$ is only $4$, this vertex must be the same
for every edge of $F$, which is not possible for a maximum degree $4$.
This excludes the first graph.

In the case of the second graph, again, because every edge must belong
to a triangle, we must have the following picture:

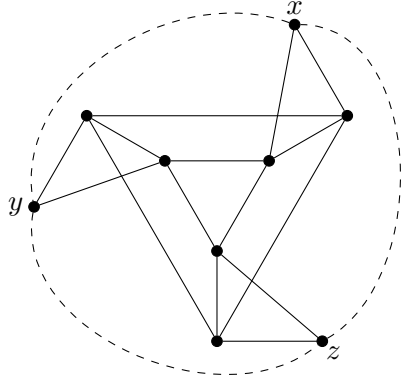
\begin{figure}[h]
  \centering
   \begin{tikzpicture}%[scale=1.2]
  \tikzstyle{every node}=[draw,circle,fill=black,minimum size=4pt,
                            inner sep=0pt]

   \foreach \x in {0,1,2}
    {
        \node (\x) at (30+120*\x:0.8cm) {};
    };
    \foreach \x in {3,4,5}
    {
        \node (\x) at (30+120*\x:2cm) {};
    };
    \draw (0) -- (1) -- (2) -- (0);
    \draw (3) -- (4) -- (5) -- (3);
    \draw (0) -- (3) (1) -- (4) (2) -- (5);

    \draw (3) -- ++(120:1.4cm) node (6) [label=above:$x$] {} -- (0);
    \draw (4) -- ++(240:1.4cm) node (7) [label=left:$y$] {} -- (1);
    \draw (5) -- ++(0:1.4cm) node (8) [label= below right:$z$] {} -- (2);

    \draw[style=dashed] (8) to [out=30,in=0] (6);
    \draw[style=dashed] (6) to [out=160,in=100] (7);
    \draw[style=dashed] (7) to [out=-100,in=-140] (8);
  \end{tikzpicture}
\end{figure}

Note that the vertices $x,y$ and $z$ must be distinct, because if two
of them are equal we obtain a set of $8$ vertices with two outgoing
edges, which contradicts $\lambda_3=6$ ($\lambda_3(X/H)=3$ translates
to $\lambda_3=6$ when we remove arc-orientation).

Finally, from the structure of $F$ where we see that every edge must
not only belong to a triangle but also to a $4$-cycle, we have that the
only possibility is for the vertices $x,y,z$ to be connected. But then
every vertex is of degree~$4$ which means that we have described the
whole graph $X/H$ which has $9$ vertices, so that $F$ is not a $4$-separable set.
\end{proof}

Together Lemmas \ref{lem:F=4}, \ref{lem:F=5} and \ref{lem:F=6}
prove Proposition~\ref{prop:lambda4}. To prove
Proposition~\ref{prop:lambda3} we are just left with the cases when $HS$
consists of either $3$ cosets or the complement of $3$ cosets.
We shall need:

\begin{lemma}\label{lem:HxHax}
  Without loss of generality, any triangle of $X/H$ is made up of
  three cosets of the form $Hx,Hax,Ha^2x$.
\end{lemma}

\begin{proof}
  Set $H=\{1,h\}$.
  Let $Hx$ be one of the cosets involved in the triangle. The coset
  $Hx$ is connected by outgoing arcs to the cosets $Hax$ and $Hahx$.
  If $Hx$ is connected to $Hahx$ in the triangle, then rename $x$ the
  group element $hx$ to obtain that the triangle always contains an
  edge of the form $Hx \rightarrow Hax$. Now the coset $Hax$ is
  connected by outgoing arcs to the cosets $Ha^2x$ and $Hahax$.
  If $Hax$ is connected to $Hahax$ in the triangle, then rename $a$
  the group element $ha$: this does not change the set $A=H\cup Ha$.
\end{proof}

\begin{proposition}\label{prop:3cosets}
  $A=H\cup Ha$ cannot be a $2$-atom of $S$ if $HS$ consists of three
  right cosets modulo $H$.
\end{proposition}

\begin{proof}
  The cosets intersected by $S$ must make up an arc $3$--atom of $X/H$
  which is a triangle by Lemma~\ref{lem:triangle} and
  of the form $Hx,Hax,Ha^2x$ by Lemma~\ref{lem:HxHax}.
  We have two cases, depending on the nature of the arc leading from
  $Ha^2x$ to $Hx$.
  \begin{enumerate}[\bf (a)]
  \item We have $a^3\in H$, in which case we must have
        $S=\{x,ax,a^2x\}$. But then $|\{1,a\}S|\leq 4 = |S| +
        |\{1,a\}| -1$, so that $A=H\cup Ha$ cannot be a $2$-atom of
        $S$.
  \item Setting $H=\{1,h\}$ we have $aha^2\in H$. We cannot have
        $aha^2=1$, otherwise $a^3\in H$ and we are back to the
        preceding case, therefore $aha^2=h$. In this case we must have
        $S=\{x,ax,ha^2x\}$. Since $haha^2=1$ we have then
        $|\{1,ha\}S|\leq 4 = |S| +
        |\{1,ha\}| -1$, so that $A=H\cup Ha$ cannot be a $2$-atom of
        $S$. \qedhere
  \end{enumerate}
\end{proof}

\begin{proposition}\label{prop:complement}
  $A=H\cup Ha$ cannot be a $2$-atom of $S$ if the complement of $HS$ consists of three
  right cosets modulo $H$.
\end{proposition}

\begin{proof}
  By Lemma~\ref{lem:triangle}, the three cosets of the complement of $HS$ must form a triangle so
  that, by Lemma~\ref{lem:HxHax}, the complement of $HS$ can be
  assumed to be equal to
  $$Hx\cup Hax \cup Ha^2x.$$
  From $HS=HaS$ we have that the complement of $S$ must be equal to:
  $$\overline{S} =
    a^{-1}(Hx\cup Hax \cup Ha^2x) \cup Hx\cup Hax \cup Ha^2x.$$
  Now since there is an arc in $X/H$ from $Ha^2x$ to $Hx$, we have
  either $a^3\in H$ or $aha^2\in H$, where we write $H=\{1,h\}$.
  \begin{enumerate}[\bf (a)]
  \item If $a^3\in H$, then clearly $\overline{S}x^{-1}ax =
    \overline{S}$, so that $Sx^{-1}ax = S$, in other words $S$ is
    stable by right multiplication by the subgroup generated by
    $x^{-1}ax$, but this contradicts $\kappa_1(S)=\kappa_1(S^{-1})=|S|-1$.
   \item If $aha^2\in H$ then we must have $aha^2=h$, because
         $aha^2=1$ implies $a^3=h$ which brings us back to the
         preceding case. Now $aha^2=h$ is equivalent to $ahaha=1$,
         which implies that the cosets $H$, $Ha$, $Haha$ make up a
         triangle in $X/H$. But this means that the four cosets
         $H,Ha,Ha^2,Haha$ form a $K_4^*$ in $X/H$. This in turn
         implies by Lemma~\ref{lem:K4*}
         that $X/H$ is an octahedron, which means that $HS$
         must consist of exactly three cosets modulo $H$. The result
         now follows from Proposition~\ref{prop:3cosets}. \qedhere
  \end{enumerate}
\end{proof}

Together, Propositions \ref{prop:lambda4}, \ref{prop:3cosets} and
\ref{prop:complement} complete the proof of
Proposition~\ref{prop:lambda3}.
Given the remark after Theorem~\ref{thm:2coset} stating that the
complement of $HS$ contains at least two cosets and the ensuing
discussion leading to Proposition~\ref{prop:lambda3}, this
in turn proves Theorem~\ref{thm:2coset}.

\section{Conclusion and Comments}

\subsection{Proof of Theorem \ref{thm:main} }\label{sec:proof}

The three possibilities in the Theorem depend on the cardinality 
and structure of the $2$--atom of $S$. 
We may assume that $\alpha_2 (S)\le \alpha_2 (S^{-1})$. 
Also, if $\kappa_1 (S)<|S|-1$ then, by Theorem \ref{thm:1atom}, 
the $1$--atom of $S$ is a subgroup $H$ leading to case {\it (ii)}. 
Thus we may assume that $\kappa_2 (S)=\kappa_1 (S)=|S|-1$.

According to Proposition \ref{prop:m+3}, either a $2$--atom $U$ of $S$ has cardinality $2$, which easily yields case {\it (i)}, or $U$ is periodic. 
Then, by Lemma \ref{lem:k2=k1} and Lemma \ref{lem:2coset}, 
either there is a $2$--fragment of $S$ which is a subgroup, giving
case {\it (ii)}, 
or there is a $2$--atom  which is the union of two right cosets of
some subgroup. In this last case, Theorem \ref{thm:2coset} gives case
{\it (iii)}:
note that equality \eqref{eq:conj} that we have used throughout
Section~\ref{sec:notsubgroups} gives the additional property
$|HaH|=|H|^2$ mentioned in case {\it (iii)}.

\subsection{Examples when {\it (i)} and {\it (ii)} in
  Theorem~\ref{thm:main} do not hold}

As it has been already mentioned, the degenerate case in Theorem
\ref{thm:main} {\it (iii)} can actually hold, without any of the two other
cases being satisfied. 
We next give an infinite family of examples which illustrate this fact. 

Let $p$ be a prime and let $q$ be an odd prime divisor of $p-1$. Let
$H_0$ be the subgroup of order $q$ of the multiplicative group of
$\zp$.
Consider the semidirect product $G = \zp \ltimes H_0$ defined by:
  $$(x,h)(y,k) = (x+hy, hk).$$ 
Set $a=(1,1)$ and $H=\{0\}\times H_0$. We have $a^{-1}=(-1,1)$ and:
\begin{align*}
  Ha      &= \{(h,h), h\in H_0\}\\
 a^{-1}H  &= \{(-1,h), h\in H_0\}\\
 a^{-1}Ha &= \{(-1+h,h), h\in H_0\}
\end{align*}
Let $A = H\cup Ha$ and
$$B= H\cup Ha\cup a^{-1}H \cup a^{-1}Ha$$
and set $S = G\setminus B$.

Since $-1\not\in H_0$ ($H_0$ has odd order) we have $a^{-1}H \cap
Ha=\emptyset$ and $a^{-1}Ha\cap H = \{(0,1)\}$. Therefore
  $$(H\cup Ha) \cap (a^{-1}H \cup a^{-1}Ha) = \{(0,1)\}$$
and $|S|=|G|- 4|H|+1$. We have $AS = G\setminus A$, hence
  $$|AS| = |S| + |A| -1.$$
Note that $B=B^{-1}$, so that $S=S^{-1}$. The subgroup $H$ does not satisfy
condition ${\it (ii)}$ of Theorem~\ref{thm:main} since we have
  $$|HS| = |S| + 2|H|-1.$$
Let us check the
other proper subgroups of $G$. Since $|G|=pq$ every proper subgroup is
cyclic and
$$(x,h)^i = (x(1+h+\cdots +h^{i-1}),h^i)$$
so that $(x,h)$ is of order $p$ when $h=1$ and of order $q$ otherwise.
There is therefore just one subgroup of order $p$ namely
$$G_p=\{(x,1),x\in\zp\}.$$
It is straightforward to check that for every $b\in B$,
$$G_pb\cap S \neq\emptyset$$
hence $G_pS=G$.
It is also straightforward to check that all subgroups of order $q$
coincide with the set of conjugate subgroups
  $$K_x = (x,1)^{-1}H(x,1) = \{(x(h-1),h), h\in H_0\}$$
where $x$ ranges over $\zp$.
It is again readily checked that for $x\neq 0,1$ and for every $b\in
B$,
$$K_xb\cap S \neq\emptyset$$
so that $K_xS=G$. For $x=0$ we have $K_0=H$ and for $x=1$,
$K_1=a^{-1}Ha$ in which case
$$|K_1S| = |S| + 2|K_1| -1.$$
Since $S^{-1}=S$ we have shown that case {\it (ii)} of
Theorem~\ref{thm:main} does not hold.

For case {\it (i)} of Theorem~\ref{thm:main} to hold, there must exist
a group element $r\neq 1$ such that $|\{1,r\}S| = |S|+1$. But this means
that we would have $|\{1,r\}B| = |B|+1$ as well, and this is readily
excluded for all $r$.

Hence only case {\it (iii)} of Theorem~\ref{thm:main} can hold for
this group $G$ and this subset $S$. 

Note that if we have $q=(p-1)/2$ (so that $q$ is a Sophie Germain
prime),
then we have $|S| = |G| - 4|H|+1 = |G|
-2\sqrt{2}|G|^{1/2}(1-O(q^{-1/2}))$.
In particular, the condition $|S|\leq |G|-4|G|^{1/2}$ in Corollary~\ref{cor:sqrt}
cannot be improved upon significantly.

\end{document}